\documentclass[12pt]{amsart}
\usepackage{amsmath,mathrsfs,amsthm, xypic,verbatim,amscd,color}
\usepackage{txfonts}
\usepackage{graphicx}
\usepackage{colordvi}
\usepackage{times}

\newtheorem{theorem}{Theorem}
\newtheorem{question}[theorem]{Question}
\newtheorem{conjecture}[theorem]{Conjecture}
\newtheorem{proposition}[theorem]{Proposition}
\newtheorem{corollary}[theorem]{Corollary}
\newtheorem{definition}[theorem]{Definition}
\newtheorem{lemma}[theorem]{Lemma}
\newtheorem{remark}[theorem]{Remark}
\newtheoremstyle{remarks}{10pt}{10pt}{ }%
{}{\bfseries}{.}{ }{}
\theoremstyle{remarks}

\DeclareMathOperator{\End}{End}

\newcommand{\gra}{\alpha}        
  
 \newcommand{\grl}{\lambda}     
      
 \newcommand{\mR}{\mathbb R} 

\newcommand{\lra}      {\longrightarrow}
\newcommand{\senza}    {\smallsetminus}

            \newcommand{\st}       {\, : \,}

         \newcommand{\mand}     {\text{ and }}
  \newcommand{\mforall}  {\text{ for all }}

\font\bigsymb=cmsy10 at 4pt
\def\bigdot{{\kern1.2pt\raise 1.5pt\hbox{\bigsymb\char15}}}

\usepackage{float}
\floatstyle{plain}
\newcommand{\g}{\mathfrak{g}}
\def\fb{{\mathfrak b}}
\def\ft{{\mathfrak t}}
\newcommand\GL{\operatorname{GL}}
\begin{document}

\title{Components of $V(\rho)\otimes V(\rho)$}

\author[Shrawan Kumar]{Shrawan Kumar \\
(With an Appendix by Rocco Chiriv\`\i\ and  Andrea Maffei)}

\maketitle
\section{Introduction}
Let $\mathfrak{g}$ be any simple Lie algebra over $\mathbb{C}$.
We fix a Borel subalgebra  $\mathfrak{b}$ and a Cartan subalgebra $\mathfrak{t} \subset \mathfrak{b}$  and let $\rho$ be the half sum of positive roots, where the roots of $\mathfrak{b}$ are called the positive roots. For any dominant integral weight $\lambda \in \mathfrak{t}^*$, let $V(\lambda)$ be the corresponding irreducible representation of $\mathfrak{g}$. B.  Kostant  initiated (and popularized) the study of the irreducible components of the tensor product  $ V(\rho)\otimes V(\rho)$. In fact, he asked (or possibly even conjectured) if the following is true.

\begin{question} (Kostant)
Let $\lambda$ be a dominant integral weight. Then, $V(\lambda)$ is a component of  $ V(\rho)\otimes V(\rho)$ if and only if 
$\lambda \leq 2\rho$ under the usual Bruhat-Chevalley order on the set of weights. 
\end{question}

It is, of course, clear that if  $V(\lambda)$ is a component of  $ V(\rho)\otimes V(\rho)$, then
$\lambda \leq 2\rho$.
 
One of  the main motivations  behind Kostant's question was his result that the exterior algebra $\wedge \mathfrak{g}$, as a $\mathfrak{g}$-module under the adjoint action, is isomorphic with $2^r$ copies of $ V(\rho)\otimes V(\rho)$, where $r$ is the rank of $\mathfrak{g}$ (cf. [Ko]). Recall that 
$\wedge \mathfrak{g}$ is the underlying space of the standard chain complex computing the homology of the Lie algebra $\mathfrak{g}$, which is, of course, an object of immense interest. 

\begin{definition} {\em An integer $d\geq 1$ is called a {\it saturation factor} for $\mathfrak{g}$,
 if for any $(\lambda,\mu,\nu)\in D^3$ such that $\lambda+\mu+\nu $ is in the root lattice and the space of $\mathfrak{g}$-invariants:
$$ [V(N\lambda)\otimes V(N\mu)\otimes 
  V(N\nu)]^{\mathfrak{g}}\neq 0$$ for some integer $N>0$, then
$$ [V(d\lambda)\otimes V(d\mu)\otimes 
  V(d\nu)]^{\mathfrak{g}}\neq 0,$$
where $D\subset \mathfrak{t}^*$ is the set of dominant integral weights of $\mathfrak{g}$. 
Such a $d$ always exists (cf. [Ku; Corollary 44]).}
\end{definition}
Recall that $1$ is a saturation factor for $\mathfrak{g}=sl_{n}$,  as proved by Knutson-Tao [KT]. By results of Belkale-Kumar [BK$_2$] (also obtained by Sam
[S]) and Hong-Shen [HS], $d$ can be taken to be  $2$ for $\mathfrak{g}$ of types  $B_r, C_r$ and  $d$ can be taken to be  $4$ for $\mathfrak{g}$ of type $D_r$  by a result of Sam [S]. As proved by Kapovich-Millson [KM$_1$, KM$_2$], the saturation factors $d$ of $\mathfrak{g}$ of types $G_2, F_4, E_6, E_7, E_8$ can be taken to be $2 \,(\text{in fact any} \,d\geq 2), 144, 36, 144, 3600$ respectively. (For a discussion of saturation factors $d$, see [Ku, $\S$10].)

Now,   the following (weaker) result is our main theorem. The proof uses 
a description of the eigencone of $\mathfrak{g}$ in terms of certain 
inequalities due to Berenstein-Sjamaar coming from the cohomology of the 
flag varieties associated to $\mathfrak{g}$, a `non-negativity'  result 
due to Belkale-Kumar  and Proposition (9) due to R. Chiriv\`\i\ and  A. 
Maffei given in the Appendix.

An interesting aspect of our work is that we make an essential use of a solution of the eigenvalue problem and  saturation results  for any $\mathfrak{g}$.

\begin{theorem}
Let $\lambda$ be a dominant integral weight such that $\lambda\leq 2\rho$. Then, $V(d\lambda)\subset V(d\rho)\otimes V(d\rho)$, where $d\geq 1$ is any saturation factor for $\mathfrak{g}$.

In particular, for $\mathfrak{g}=sl_{n}$, $V(\lambda)\subset V(\rho)\otimes V(\rho)$.
\end{theorem}
\noindent
{\bf Acknowledgements.} I thank Corrado DeConcini who brought to my attention  Question (1). Partial support from NSF grant number
DMS- 1501094 is gratefully acknowledged.

\section{Proof of Theorem (3)}
We now prove Theorem (3).
\begin{proof} 
Let $\Gamma_{3}(\mathfrak{g})$ be the {\it saturated tensor semigrou}p defined by
$$
\Gamma_{3}(\g)=\{(\lambda,\mu,\nu)\in D^{3}:[V(N\lambda)\otimes V(N\mu)\otimes V(N\nu)]^{G}\neq 0\text{~~ for some~~ } N>0\}.
$$
To prove the theorem,  it suffices to prove that $(\rho,\rho,\lambda^{*})\in \Gamma_{3}(G)$, where $\lambda^{*}$ is the dual weight $-w_o\lambda$,
$w_o$ being the longest element of the Weyl group of $\mathfrak{g}$.
Let $G$ be the connected, simply-connected  complex algebraic
group with Lie algebra $\g$. Let $B$ (resp. $T$) be the  Borel subgroup (resp. maximal torus) of $G$ with Lie algebra $ \fb$ (resp. 
$ \ft$).  Let $W$ be the Weyl group of $G$. For any standard parabolic subgroup $P\supset B$ with Levi subgroup $L$ containing $T$, let $W^P$ be  the set of smallest length coset representatives in $W/W_L$, $W_L$ being the Weyl group of $L$. Then, we have the 
Bruhat decomposition:
$$G/P=\sqcup_{w\in W^P}\, \Lambda_w^P,\,\,\,\text{where}\,\, \Lambda^P_w:= BwP/P.$$
Let $\bar{\Lambda}_w$ denote the closure of $\Lambda_w$ in $G/P$.  We denote by $[\bar{\Lambda}_w]$ the Poincar\'e dual of its
fundamental class. Thus, $[\bar{\Lambda}_w]$ belongs to the singular cohomology:
$$[\bar{\Lambda}_w]\in H^{2(\text{dim}\,G/P- \ell(w))}(G/P, \mathbb{Z}),$$
where $\ell(w)$ is the length of $w$.

Let $\{x_{j}\}_{1\leq j\leq r} \subset \mathfrak{t}$ be the  dual to the simple roots $\{\alpha_{i}\}_{1\leq i\leq r}$, i.e., 
$$\alpha_{i}(x_{j})=\delta_{i,j}.$$ 

In view of [BS] (or [Ku; Theorem 10]), it suffices to prove that for any standard maximal parabolic subgroup $P$ of $G$ and triple $(u,v,w)\in (W^{P})^{3}$ such that the cup product of the corresponding  Schubert classes in $G/P$~:
\begin{equation}
[\bar{\Lambda}^{P}_{u}]\cdot [\bar{\Lambda}^{P}_{v}]\cdot [\bar{\Lambda}^{P}_{w}]=k[\bar{\Lambda}^{P}_{e}]\in H^{*}(G/P, \mathbb{Z}), \,\,\,\text{for some} \,\, k\neq 0, \label{eq0}
\end{equation}
the following inequality is satisfied:
\begin{equation}
\rho(ux_{P})+\rho(vx_{P})+\lambda^{*}(wx_{P})\leq 0.\label{eq*}
\end{equation}
Here, $x_{P}:=x_{i_{P}}$, where $\alpha_{i_{P}}$ is the unique  simple root not in the Levi of $P$. 

Now, by  [BK$_1$; Proposition 17(a)] (or [Ku; Corollary 22 and Identity (9)]), for any $u,v,w\in (W^P)^3$ such that  the equation \eqref{eq0}  is satisfied,
\begin{equation}
(\chi_{w_{o}ww_{o}^P}-\chi_{u}-\chi_{v})(x_{P})\geq 0,\label{eq1}
\end{equation}
where $w^{P}_{o}$ is the longest element in the Weyl group of $L$ and 
$$
\chi_{w}:=\rho-2\rho^{L}+w^{-1}\rho 
$$
($\rho^{L}$ being the half sum of positive roots in the Levi of $P$).

Now,
\begin{align}
&(\chi_{w_{o}ww_{o}^P}-\chi_{u}-\chi_{v})(x_{P})\notag\\
&=({\rho}-w^{P}_{o}w^{-1}\rho-{\rho}-u^{-1}\rho-\rho-v^{-1}\rho)(x_{P}),\text{~~ since~ } \rho^{L}(x_{P})=0\notag\\
&= (-\rho-u^{-1}\rho-v^{-1}\rho-w^{-1}\rho)(x_{P}),\text{~~ since~ }w^{P}_{o}(x_{P})=x_{P}.\label{eq2}
\end{align}
Combining \eqref{eq1} and \eqref{eq2}, we get
\begin{equation}
(\rho+u^{-1}\rho+v^{-1}\rho+w^{-1}\rho)(x_{P})\leq 0\,,\text{~~ if \eqref{eq0} is satisfied.}\label{eq3}
\end{equation}
We next claim that for any dominant integral weight $\lambda \leq 2 \rho$ and any $u,v,w\in (W^P)^3$,
\begin{equation}
\rho(ux_{P})+\rho(vx_{P})+\lambda^{*}(wx_{P})\leq (\rho+u^{-1}\rho+v^{-1}\rho+w^{-1}\rho)(x_{P}),\label{eq4}
\end{equation}
which is equivalent to
\begin{equation}
\lambda^{*}(wx_{P})\leq (\rho+w^{-1}\rho)(x_{P}).\label{eq5}
\end{equation}

Of course \eqref{eq3} and \eqref{eq4} together give \eqref{eq*}. So, to prove the theorem, it suffices to prove \eqref{eq5}. Since the assumption on $\lambda$ in the theorem is invariant under the transformation $\lambda\mapsto \lambda^{*}$, we can replace $\lambda^{*}$ by $\lambda$ in \eqref{eq5}.
By Proposition (9) in the appendix, $\lambda=\rho+\beta$, where $\beta$ is a weight of $V(\rho)$ (i.e.,  the weight space of $V(\rho)$ corresponding to the weight $\beta$ is nonzero). Thus, 
$$
\lambda(wx_{P})=\rho(wx_{P})+\beta(x_{P}),\text{~~ for some weight $\beta$ of $V(\rho)$}.
$$

Hence,
$$
\lambda(wx_{P})=\rho(wx_{P})+\beta(x_{P})\leq (w^{-1}\rho + \rho)(x_{P}).
$$

This establishes \eqref{eq5} and hence the theorem is proved.
\end{proof}

We recall the following conjecture due to Kapovich-Millson [KM$_1$] (or [Ku; Conjecture 47]).
\begin{conjecture} Let $\mathfrak{g}$ be a simple, simply-laced Lie algebra over $\mathbb{C}$. Then, $d=1$ is a saturation factor for
$\mathfrak{g}$.
\end{conjecture}
The following theorem follows immediately by combining Theorem (3) and Conjecture (4).
\begin{theorem} For any  simple, simply-laced Lie algebra $\mathfrak{g}$ over $\mathbb{C}$, assuming the validity of Conjecture (4), Question (1) has an affirmative answer for $\mathfrak{g}$, i.e., for any dominant integral weight $\lambda \leq 2\rho$,  $V(\lambda)$ is a component of  $ V(\rho)\otimes V(\rho)$.

Thus, assuming the validity of Conjecture (4), Question (1) has  an affirmative answer for any simple $\mathfrak{g}$ of type 
$D_r (r\geq 4); E_6; E_7;$ and $E_8$ as well (apart from $\mathfrak{g}$ of type $A_r$ as in Theorem (3)).
\end{theorem}
\begin{remark}
{\em  By an explicit calculation using the program LIE, it is easy to see that Question (1) has an affirmative answer for 
simple $\mathfrak{g}$ of types $G_2$ and $F_4$ as well.}
\end{remark}

\newpage
 
\section{Appendix (due to R. Chiriv\`\i\ and A. Maffei)}

We follow the notation and assumptions from the Introduction. In particular, $\mathfrak{g}$ is a simple Lie algebra over $\mathbb{C}$.
Let $\{\omega_i\}_{i\in I}$ be the fundamental weights,  $\{\gra_i\}_{i\in I}$ the simple roots, and  $\{s_i\}_{i\in I}$ 
the simple reflections, where $I:=\{1\leq i \leq r\}$. For any  $J\subset I$, let
 $W_J$ be the parabolic subgroup of the Weyl group $W$ generated by
$s_j$ with $j\in J$ and let $\Phi_J$ be the root system generated by the
simple roots $\gra_j$ with $j\in J$. Set 
$$\Omega :=\bigoplus_{i\in
  I} \mR \omega_i;\,\,\Omega_J :=\bigoplus_{j\in J} \mR \omega_j ,$$
and let $ \pi_J:\Omega\lra \Omega_J $ be  the projection with kernel
$\Omega_{I\senza J}$. The projection $\pi_J(\Phi_J)$ of the roots in
$\Phi_J$ gives a root system whose fundamental weights are given by
$\{\omega_j\st j\in J\}$.

\medskip

Let $A\subset \mathfrak{t}^*$ be the dominant cone, $B\subset \mathfrak{t}^*$ the cone generated by $\{-\gra_i:i\in I\}$
and $C:=2\rho+B$.  We want to describe the vertices of the
polytope $A\cap C$. For $J\subset I$ define 
$$
A_J:=\mR_{\geq 0}[\omega_j\st j\in J] ,\,\,\, B_J :=\mR_{\geq 0}[-\gra_j\st j\in J] \; \mand \; C_J:=2\rho+B_J .
$$ 

The sets $A_J$ and $B_J$ are the faces of $A$ and $B$.
The vertices of the polytope $A\cap C$ are given by the zero
dimensional nonempty intersections of the form $A_J\cap C_H$.

For any $J\subset I$, let $b_J:=\sum_{\gra\in\Phi^+_J}\gra$ and $c_J:=2\rho-b_J$. 
All these points are different. Moreover, $c_I=0$ and $c_\emptyset=2\rho$.

\begin{lemma}
For each $J\subset I$, we have
$$
A_{I\senza J}\cap C_J=\{c_J\}.
$$
Moreover, none of  the other intersections $A_H\cap C_K$  give a single point.
\end{lemma}
\begin{proof}
Observe  that 
$$b_J=2\sum_{j\in J}\omega_j + \sum_{\ell\notin
J}a_\ell \omega_\ell, \,\,\,\text{where}\,\, a_\ell \leq 0.$$ Hence, $c_J\in A_{I\senza J}\cap C_J$.

Consider now an intersection of the form $A_{I\senza H}\cap
C_K$. Assume it is not empty and that $y=2\rho - x \in A_{I\senza H}\cap C_K$. Then, 
$x=2\sum_{h\in H}\omega_h +\sum_{\ell\notin
H}a'_\ell \omega_\ell$. Now, notice that if $h\notin K$, the
coefficient of $\omega_h$ in $x$ can not be positive. So, we must have
$K\supset H$.  If $K\supset H$ and $K\neq H$, then
$$ A_{I\senza H}\cap C_K \supset \big(A_{I\senza H}\cap
C_H\big)\cup\big(A_{I\senza K}\cap C_K\big)\supset \{c_H,c_K\}.
$$
Hence,  it is not a single point. 

It remains to prove that $A_{I\senza J}\cap C_J\subset\{c_J\}.$ Let
$y=2\rho-x$ as before.  Notice that $\pi_J(x) = 2\sum_{j\in
J}\omega_j$ and $\pi_J(x)=\sum_{\gra\in\Phi_J^+}\pi_J(\gra)$. Since $\pi_J$ is
injective on $B_J$, we must have $x=b_J$ and the claim follows.
\end{proof}

We have the following Corollary.

\begin{corollary}
The intersection $A\cap C$ is the convex hull of the points $\{c_J\st J\subset I\}$.
\end{corollary}

We now prove the following main result of this Appendix.

\begin{proposition}  Let $\lambda\leq 2 \rho$ be  a dominant integral weight. Then, 
$$\lambda=\rho + \beta,$$
for some weight $\beta$  of  $V(\rho)$. 
\end{proposition}
\begin{proof} 
Let $Q\subset \mathfrak{t}^*$ be the root lattice (generated by the simple roots) and let $H_\rho$ be the
convex hull of the weights $\{w(\rho)\st w\in W\}$.  Recall that
the weights of the module $V (\rho)$ are precisely  the elements of the
intersection
$$(\rho+Q)\cap H_\rho.$$ If $\grl$ is as  in the proposition, then it is
clear that $\grl-\rho\in \rho+Q$. So, we need to prove that it belongs
to $H_\rho$. To check this, it is enough to check that $(A\cap
C)-\rho\subset H_\rho$ or equivalently that
$$ c_J-\rho \in H_\rho , \;\mforall\; J\subset I. $$ We 
have $$c_J-\rho=\rho-b_J=w^J_o(\rho)\in H_\rho ,$$ where $w^J_o$ is the longest
element in the parabolic subgroup $W_J$.  Indeed, to prove the last equality,
observe that $\rho-w^J_o(\rho)$
is a sum of roots $\gra_j$ with $j\in J$. So, since $\pi_J$ is
injective on $B_J$, it is enough to check that
$\pi_J(\rho-w^J_o(\rho))=\pi_J(b_J)$. Hence, we are reduced to study the
case in which $J=I$, for which we have $w^I_o(\rho)=-\rho$ and
$\rho-w^I_o(\rho)=2\rho=b_I$.
\end{proof}

\vskip5ex

Addresses:
\vskip1ex
Shrawan Kumar: Department of Mathematics, University of North Carolina, Chapel Hill, NC 27599-3250, USA (shrawan@email.unc.edu)
\vskip1ex

Rocco Chiriv\`\i:  Department of Mathematics and Physics, 
Universit\`a del Salento, Lecce, Italy (rocco.chirivi@unisalento.it)
\vskip1ex
Andrea Maffei: Dipartimento di Matematica, Universit\`a di Pisa, Pisa, Italy (maffei@dm.unipi.it)

\end{document}